\documentclass{amsart}

\DeclareMathOperator{\Ad}{Ad}

\DeclareMathOperator{\Int}{Int}

\DeclareMathOperator{\Gal}{Gal} 
\DeclareMathOperator{\ab}{ab}
\DeclareMathOperator{\sgn}{sgn}

\DeclareMathOperator{\res}{res}

\DeclareMathOperator{\inv}{inv} 
\DeclareMathOperator{\new}{new} 
\DeclareMathOperator{\Norm}{Norm}
\DeclareMathOperator{\obs}{obs} 
\DeclareMathOperator{\Tam}{Tam}

\usepackage{amscd}
\usepackage{amssymb}

\numberwithin{equation}{subsection}

\newtheorem{theorem}{Theorem}[subsection]

\newtheorem{lemma}[theorem]{Lemma}
\newtheorem{proposition}[theorem]{Proposition}

\theoremstyle{definition}

\newtheorem{remark}[theorem]{Remark}

\begin{document}
 
\title[Splitting invariants and sign conventions]
{On splitting invariants and sign conventions in endoscopic transfer}

\author[R. Kottwitz]{R. Kottwitz}
\address{R. Kottwitz\\Department of Mathematics\\ University of
Chicago\\ 5734 University Avenue\\ Chicago, Illinois 60637}
\email{kottwitz@math.uchicago.edu}

\author[D. Shelstad]{D. Shelstad}
\address{D. Shelstad\\Mathematics Department\\ Rutgers University\\ Newark,
NJ 07102}
\email{shelstad@rutgers.edu}


\begin{abstract} 
The transfer factors for standard endoscopy involve, among other things, the
Langlands-Shelstad splitting invariant. This note introduces a twisted
version of that splitting invariant. The twisted splitting invariant is
then   used to define a better twisted factor 
$\Delta_I$. In addition we correct a sign error in the definition of twisted
transfers. There are two ways to correct the sign error. One way yields
twisted transfer factors $\Delta'$ that are compatible with the classical
Langlands correspondence. The other way yields twisted transfer factors
$\Delta_D$ that are compatible with a renormalized version of the Langlands
correspondence. 
\end{abstract}

\maketitle

\section{Introduction}

Waldspurger has observed that in order to have smooth matching \cite{W1} 
of functions for twisted endoscopy, the definition of twisted transfer 
factors given in \cite{KS} must be modified if the attached restricted root
system is non-reduced. This happens only when the
root system itself has an irreducible component of type $A_{2n}$ such that 
\begin{enumerate}
\item some power $\theta^i$ of $\theta$ preserves that irreducible
component, and 
\item $\theta^i$ acts nontrivially on the Dynkin diagram of that component. 
\end{enumerate} 
The modification is of course needed when $\theta$ is, up to inner
automorphism, transpose-inverse on $GL(2n+1)$. On the other hand, it never
arises in the context of cyclic base change. 

 Waldspurger  has made a specific proposal \cite{W2} for modifying the
definition of twisted transfer factors and has shown that his modified
factors do yield smooth matching for twisted endoscopy over $p$-adic fields.
Waldspurger proposes to modify the term $\Delta_{II}$ by replacing the
expression  (4.3.4) in \cite{KS}, namely 
\begin{equation}\label{eq.deltaKS}
\chi_{\alpha_{\res}}(N\alpha(\delta^*)+1),
\end{equation}
 by the slightly different expression  
\begin{equation}\label{eq.deltaW}
\chi_{\alpha_{\res}}((N\alpha(\delta^*)+1)/2).
\end{equation} 
When $2$ is a nonzero square in the ground field $k$ (for example, when 
$k$ is $\mathbb R$ or $\mathbb C$) the terms \eqref{eq.deltaKS} and
\eqref{eq.deltaW} are equal, but in general they can certainly be
different.  

It might be thought that nothing further need be said, since
Waldspurger's modification yields a satisfactory theory for all local fields
of characteristic
$0$. The trouble is that \eqref{eq.deltaW} is undefined for local fields of
characteristic $2$. For such local fields it seems unlikely that there is
any way to fix the definition of twisted transfer factors by modifying
$\Delta_{II}$. However, twisted transfer factors are the product of four
terms, which leaves
open the possibility of modifying one of the three other terms in such a way
as to obtain the same overall result, i.e., to multiply the twisted
transfer factor of \cite{KS} by the sign 
\[
\prod_\beta \sgn_{F_\beta/F_{\pm\beta}}(2),
\]
where the product is taken over a set of representatives $\beta$ for the
symmetric
$\Gamma$-orbits in the set of restricted roots
$\beta$ that are of type
$R_3$ and come from $H$.  

In this note we will show how this can be done. We will leave $\Delta_{II}$
unchanged and instead  modify $\Delta_{I}$. This approach seems quite
natural.  
The term $\Delta_I$ of \cite{KS} was defined in terms of the
Langlands-Shelstad splitting invariant $\lambda(T^x) \in H^1(k,T^x)$. The
torus $T^{x}$ is defined  after passage to the simply-connected cover
of the derived group of $G.$ It is harmless to assume that $G$ itself is
semisimple and simply-connected. Then $T^{x}=T^{\theta }$ and so we have 
$ \lambda (T^{\theta })\in H^{1}(k,T^{\theta }).$ We will
define a new version
$\Delta_I^{\new}$ of
$\Delta_I$ by introducing a \emph{twisted} splitting invariant
$\lambda(T,\theta) \in H^1(k,T^\theta)$. The basic idea is quite simple:
$\lambda(T,\theta)$ is just a \emph{refinement} of the untwisted
splitting invariant $\lambda(T) \in H^1(k,T)$ of \cite{LS} obtained from
$\theta$-invariant $a$-data, in the sense that the image of
$\lambda(T,\theta)$ under
$H^1(k,T^\theta) \to H^1(k,T)$ is equal to the Langlands-Shelstad splitting
invariant $\lambda(T)$. 

Waldspurger also found a sign error in the definition of twisted transfer
factors. Contrary to what is stated in \cite{KS}, the factor
$\Delta=\Delta_I\Delta_{II}\Delta_{III}\Delta_{IV}$ proposed there 
is not independent of the choice of $\chi$-data, because changing the
choice of $\chi$-data multiplies $\Delta_{II}$ and $\Delta_{III}$ by the
same factor. In order that twisted transfer factors be independent of
$\chi$-data, either
$\Delta_{II}$ or $\Delta_{III}$ needs to occur with exponent $-1$. Since
$\Delta_I$ and $\Delta_{III}$ are linked together for other reasons, one
cannot invert $\Delta_{III}$ without inverting $\Delta_I$ at the same time. 
These considerations suggest two possible corrected versions $\Delta_D$ and
$\Delta'$ of twisted transfer factors, namely 
\begin{align}
\Delta_D:=&\Delta_I^{\new}\Delta_{II}^{-1}\Delta_{III}\Delta_{IV}, \\
\Delta':=&(\Delta_I^{\new}\Delta_{III})^{-1}\Delta_{II}\Delta_{IV}.
\end{align} 
The factors $\Delta'$ are compatible with the classical Langlands
correspondence. The factors $\Delta_D$ are compatible with the renormalized
Langlands correspondence discussed in section \ref{sec.ren}. 

This note is organized as follows. In section \ref{sec.2} we define the 
twisted  splitting invariant. In section  \ref{sec.3} we define and study the
improved version
$\Delta_{I}^{\new}$ of $\Delta_I$. In section \ref{sec.ren} we discuss the
renormalized version of the local Langlands correspondence. In section
\ref{sec.55} we define corrected versions $\Delta_D$ and $\Delta'$ of
twisted transfer factors. Subsection \ref{sub.Wh} treats the
Whittaker normalization of twisted transfer factors. Subsection
\ref{sub.cbc} relates the twisted transfer factors $\Delta'$ for cyclic base
change  to the transfer factors $\Delta'$ for standard endoscopy; this
corrects the slightly flawed treatment given in \cite{App}.  

It remains to thank Waldspurger for pointing out the described errors in
\cite{KS} and for observing that the factor $\Delta'$ works.  

\section{Definition of the twisted splitting invariant}\label{sec.2}
\subsection{Review of the Langlands-Shelstad splitting invariant} Since our
twisted splitting invariant will be  a refinement of the one in
\cite{LS}, our first task is to review the relevant constructions from that
article, whose notation we adopt almost without change. 

We work over an arbitrary ground field $k$. We do not assume that $k$ has
characteristic $0$. We choose a separable closure $\bar k$ of $k$ and put
$\Gamma=\Gal(\bar k/k)$. We consider a connected reductive group $G$ over
$k$. We assume that $G$ is quasi-split over $k$. It is convenient, and
harmless for our purposes, to assume further that $G$ is semisimple and
simply-connected. We fix a $k$-splitting $(\mathbf B,\mathbf
T,\{X_\alpha\})$ of $G$. We denote by $\sigma_{\mathbf T}$ the action of
$\sigma \in \Gamma$ on
$\mathbf T$ and   set
$\mathbf{\Gamma}=\{\sigma_{\mathbf T}:\sigma \in \Gamma\}$. We write $\mathbf
\Omega$ for the Weyl group $\Omega(G,\mathbf T)$. 

For each simple root $\alpha$ of $\mathbf T$ we denote by $M_\alpha$ the
Levi subgroup of $G$ containing $\mathbf T$ and having root system
$\{\pm\alpha\}$; the group $M_\alpha$ and its derived group $G_\alpha$ are
defined over $\bar k$. Now $G_\alpha$ is isomorphic to $SL_2$ by virtue of
our assumption that $G$ is semisimple and simply-connected. In fact there
exists a unique $\bar k$-isomorphism $\xi_\alpha:SL_2 \to G_\alpha$ such
that 
\begin{enumerate}
\item $\xi_\alpha$ maps the diagonal subgroup of $SL_2$ isomorphically to the
maximal torus $\mathbf T \cap G_\alpha$ of $G_\alpha$,  
\item $\xi_\alpha$ maps the upper triangular Borel subgroup of $SL_2$
isomorphically to the Borel subgroup $\mathbf B \cap G_\alpha$ of $G_\alpha$,
\item $\xi_\alpha$ maps 
$\begin{bmatrix}
0 & 1 \\
0 & 0
\end{bmatrix}$ to the root vector $X_\alpha$ occurring in our chosen
$k$-splitting.
\end{enumerate} 
The element $n(\alpha) \in \Norm(\mathbf T,G)(\bar k)$ obtained as the image
under $\xi_\alpha$ of 
$
\begin{bmatrix}
0 & 1 \\
-1 &0
\end{bmatrix}
$
lifts the simple reflection $\omega(\alpha) \in \mathbf\Omega$. 
For any $\omega \in \mathbf\Omega$ one obtains a lifting $n(\omega)\in
\Norm(\mathbf T,G)(\bar k)$ of $\omega$ by putting 
\[
n(\omega)=n(\alpha_1)\cdots n(\alpha_r)
\] 
for any reduced expression $\omega=\omega(\alpha_1)\cdots\omega(\alpha_r)$. 

With this notation in place we are ready to review the construction in
section 2.3 of \cite{LS}. We consider a maximal $k$-torus $T$ in $G$. We fix
$a$-data $\{a_\alpha\}_{\alpha \in R(G,T)}$ for the root system $R(G,T)$ of
$T$. Thus the elements $a_\alpha \in \bar k$ satisfy the conditions 
\begin{enumerate}
\item $a_{\sigma\alpha}=\sigma(a_\alpha)$ for all $\sigma \in \Gamma$, and
\item $a_{-\alpha}=-a_{\alpha}$. 
\end{enumerate} 

In order to define the splitting invariant
$\lambda_{\{a_\alpha\}}(T) \in H^1(k,T)$, we begin by
choosing a Borel subgroup $B$ of $G$ (over $\bar k$) that contains $T$, as
well as an element $h \in G(\bar k)$ such that $(B,T)^h=(\mathbf B,\mathbf
T)$. Denote by $\sigma_T$ both the action of $\sigma\in\Gamma$ on $T$ and its
transport to
$\mathbf T$ by $\Int(h^{-1})$. We then have 
\begin{equation}
\sigma_T=\omega_T(\sigma)\rtimes \sigma_{\mathbf T} \in \mathbf \Omega 
\rtimes \mathbf \Gamma
\end{equation}
where $\omega_T(\sigma)$ is the image in $\mathbf\Omega$ of the element
$h^{-1}\sigma(h) \in \Norm(\mathbf T,G)(\bar k)$. We use $\Int(h^{-1})$ to
transport our $a$-data from $T$ to $\mathbf T$. 

Now $\mathbf\Omega\rtimes\mathbf\Gamma$ is a group of automorphisms of
$R(G,\mathbf T)$. For any automorphism $\zeta$ of $R(G,\mathbf T)$ we
consider the element $x(\zeta) \in \mathbf T(\bar k)$ defined by 
\begin{equation}\label{eq.zeta}
x(\zeta)=\prod_{\alpha \in R(\zeta)} a_\alpha^{\alpha^\vee},
\end{equation}
where $R(\zeta):=\{\alpha \in R(G,\mathbf T) : \alpha >0,\zeta^{-1}\alpha
<0\}$. 
Here $\alpha > 0$ means that $\alpha$ is a root of $\mathbf T$ in $\mathbf
B$.  In \cite{LS} it is shown that 
\begin{equation}\label{eq.m}
m(\sigma):=x(\sigma_T)n(\omega_T(\sigma))
\end{equation}
is a $1$-cocycle of $\Gamma$ in $\Norm(\mathbf T,G)$ and that 
\begin{equation}\label{eq.t}
t(\sigma):=hm(\sigma)\sigma(h^{-1})
\end{equation} 
is a $1$-cocycle of $\Gamma$ in $T$. The class in $H^1(k,T)$ of the 
$1$-cocycle $t(\sigma)$ is by definition the splitting invariant
$\lambda_{\{a_\alpha\}}(T)$; it is independent of the choice of $h$.  It
depends on the chosen
$k$-splitting of
$G$, even though this dependence is not reflected in the notation. 

However, $\lambda_{\{a_\alpha\}}(T)$ does not depend on the choice of Borel
subgroup $B$ containing $T$ \cite[2.3.3]{LS}. We need to recall why this is
so.  Suppose that $B$ is replaced by $B'=vBv^{-1}$ with $v \in \Norm(
T,G)$. Set $u=h^{-1}vh \in \Norm (\mathbf T,G)$ and let $\mu$ be the image
of $u$ in $\mathbf\Omega$. We may as well choose $v$ so that $u=n(\mu)$.
Obviously the element $h'=vh$ satisfies $(B',T)^{h'}=(\mathbf B,\mathbf T)$. 
Using $B'$, $h'$ in place of $B$, $h$, we obtain a $1$-cocycle $t'(\sigma)$ 
of  $\Gamma$ in $T$, and in \cite{LS} it is shown that $t'(\sigma)$ is
cohomologous to $t(\sigma)$. In fact, from the proofs of Lemmas 2.3.A and
2.3.B of \cite{LS} it is clear, with our choice of $v$, that $t'(\sigma)$ is
the product of
$t(\sigma)$ and the coboundary of the element 
\begin{equation}\label{eq.cob}
h x(\mu)h^{-1} \in T(\bar k). 
\end{equation} 
This completes our review of  the untwisted splitting
invariant. 

\subsection{Definition of twisted splitting invariants} 
We retain all the previous notation and assumptions, but now we further
consider a $k$-automorphism $\theta$ of $G$ that preserves the $k$-splitting
$(\mathbf B,\mathbf T,\{X_\alpha\})$. We need to understand how $\theta$
interacts with the constructions made in the untwisted case. 

By
assumption $\theta$ preserves our chosen $k$-splitting of $G$. It follows
easily that 
\begin{equation}\label{eq.ntheq}
n(\theta(\omega))=\theta(n(\omega))
\end{equation}
for all $\omega \in \mathbf\Omega$. 

In this twisted situation we are  interested exclusively in
\emph{$\theta$-admissible} maximal
$k$-tori $T$, by which we mean that $\theta(T)=T$ and that there exists a
Borel subgroup $B$ (over $\bar k$) containing $T$ and satisfying
$\theta(B)=B$. For such $T$ the automorphism $\theta$ acts on
$R(T,G)$, and we are only interested in \emph{twisted $a$-data} for $T$, by
which we mean $a$-data for the root system $R(T,G)$ that satisfy the
additional condition 
\begin{equation}
a_{\theta(\alpha)}=a_\alpha
\end{equation}
for all $\alpha \in R(T,G)$. In other words, twisted $a$-data is nothing but
standard $a$-data that happens to be invariant under the obvious action of
$\theta$ on the set of all standard $a$-data. 

The first step in defining the untwisted splitting invariant was to
choose $B$ and $h$. In the twisted situation we begin by choosing a Borel
subgroup $B$ containing $T$ such that $\theta(B)=B$. 

\begin{lemma}
There exists $h \in G^\theta(\bar k)$ such that $(B,T)^h=(\mathbf B,\mathbf
T)$. 
\end{lemma}
\begin{proof}
Steinberg \cite{St} proved this for algebraically closed fields.  Fortunately
his proof carries over to the case of separably closed fields. 
\end{proof}
 
We now choose $h$ as in the lemma. We use $\Int(h^{-1})$ to
transport $\sigma_T$ to $\mathbf T$, and we define $\omega_T(\sigma)$ as
before.   We use our chosen twisted $a$-data to form  the elements
$x(\zeta)$ in equation \eqref{eq.zeta}.  

\begin{lemma}\label{lem.fix}
Let $\zeta$ be an automorphism of $R(G,\mathbf T)$ that commutes with the
natural action of $\theta$ on $\mathbf T$. Then the  element $x(\zeta)\in
\mathbf T(\bar k)$ lies in $\mathbf T^\theta$. In particular this is so for
the automorphisms
$\sigma_T$ of $\mathbf T$, and also for the automorphisms induced by 
elements in $\mathbf \Omega^\theta$. 
\end{lemma} 
\begin{proof} 
 We claim that the automorphism $\theta$ preserves
$R(\zeta)$.  This, together with the invariance property
$a_{\theta(\alpha)}=a_\alpha$ of our twisted $a$-data, proves that
$\theta(x(\zeta))=x(\zeta)$. The claim holds because (i) $\zeta$ commutes
with
$\theta$, and (ii) $\theta$ preserves the set of positive roots of $\mathbf
T$.   
\end{proof} 

We use $h$ and our twisted $a$-data to form the $1$-cocycle $t(\sigma)$ of
$\Gamma$ in $T$ (see \eqref{eq.t}). 

\begin{proposition}
The $1$-cocycle $t(\sigma)$  takes values in the subtorus
$T^\theta$. Its class in $H^1(k,T^\theta)$ is independent of the choice of
$B$ and
$h$.   
\end{proposition}

\begin{proof}
 That $T^\theta$ is connected and hence a subtorus of $T$
follows from our assumption that $G$ is semisimple and simply-connected. 

Next we show that $t(\sigma)$ is fixed by $\theta$. 
It follows from the previous lemma that $x(\sigma_T)$ is fixed by $\theta$. 
Since $\sigma_T$ and $\sigma_{\mathbf T}$ both commute with $\theta$, so too
does $\omega_T(\sigma)$, and it then  follows from equation \eqref{eq.ntheq}
that 
$n(\omega_T(\sigma))$ is fixed by $\theta$. We conclude that
$m(\sigma)=x(\sigma_T)n(\omega_T(\sigma))$ is also fixed by $\theta$.
Finally, it follows from \eqref{eq.t} that $t(\sigma)$ is  fixed by
$\theta$.

It is obvious that the cohomology class of $t(\sigma)$ is independent of the
choice of $h$. That it is also independent of the choice of $\theta$-stable
$B$ containing $T$ follows from the fact the element $hx(\mu)h^{-1}$ 
occurring  in \eqref{eq.cob}  is
fixed  by $\theta$ when $\mu \in \mathbf\Omega^\theta$, and this too is a 
consequence of the previous lemma. 
\end{proof}

 The class of $t(\sigma)$ in $H^1(k,T^\theta)$ is the desired twisted
splitting invariant and will be denoted by
$\lambda_{\{a_\alpha\}}(T,\theta)$.

\section{Comparison of $\Delta^{\new}_I$ with $\Delta_I$}\label{sec.3}

 As in the previous
section $G$ is semisimple and simply-connected, and $T$ is a 
$\theta$-admissible maximal $k$-torus in $G$. In this situation $T^\theta$ is
connected and hence a subtorus of $T$. We fix a $\theta$-stable Borel
subgroup (over $\bar k$) containing $T$. As positive system $R^+(G,T)$ we
take the roots of $T$ in $B$. For $\alpha \in R(G,T)$ we denote by
$\alpha_{\res}$ the restriction of $\alpha$ to $T^\theta$, and we put 
$R_{\res}(G,T)=\{\alpha_{\res}: \alpha \in R(G,T)\}$. We denote by
$\pi:R(G,T) \twoheadrightarrow R_{\res}(G,T)$ the map $\alpha \mapsto
\alpha_{\res}$. 

\subsection{Review of some results of Steinberg} We remind the reader of the
following results of Steinberg concerning the relation between
$R_{\res}(G,T)$ and $R(G,T)$ for $\theta$-admissible $T$. The
results are purely root-theoretic and so the characteristic of the ground
field plays no role here. 

\begin{theorem}[Steinberg] \hfill 
\begin{enumerate}
\item $R_{\res}(G,T)$ is a root system in $X^*(T^\theta)$, possibly
non-reduced, whose Weyl group will be denoted by $\Omega_{\res}(G,T)$. 
\item $R^+_{\res}(G,T):=\{\alpha_{\res}:\alpha \in R^+(G,T)\}$ is a positive
system in  $R_{\res}(G,T)$. 
\item The set of simple roots in $R_{\res}(G,T)$ is the image under $\pi$ of
the set of simple roots in $R(G,T)$. The set of simple roots in $R(G,T)$ is
the preimage under $\pi$ of the set of simple roots in $R_{\res}(G,T)$. 
\item The map $\pi$ induces a bijection from the set of orbits of $\theta$
in $R(G,T)$ to the set $R_{\res}(G,T)$. 
\item There is a unique homomorphism $\Omega_{\res}(G,T) \to \Omega(G,T)$
for which the restriction map $X^*(T) \twoheadrightarrow X^*(T^\theta)$ is
$\Omega_{\res}(G,T)$-equivariant. This homomorphism identifies
$\Omega_{\res}(G,T)$ with $\Omega(G,T)^\theta$. 
\item Let $\beta$ be a simple root in $R_{\res}(G,T)$. We consider the
subset $\mathbb Z\beta \cap R_{\res}(G,T)$, the intersection being taken in 
$X^*(T^\theta)$, and we denote by $M_\beta$ the unique Levi subgroup of $G$
over $\bar k$ that contains $T$ and has root system $\pi^{-1}(\mathbb Z \beta
\cap R_{\res}(G,T))$. Since this preimage is stable under $\theta$, so too is
$M_\beta$. Under the isomorphism $\Omega_{\res}(G,T) \simeq
\Omega(G,T)^\theta$ the simple reflection $\omega(\beta) \in
\Omega_{\res}(G,T)$ corresponds to the longest element in the Weyl group
$\Omega(M_\beta,T)$. When $2\beta$ is not a restricted root, the Dynkin
diagram of $M_\beta$ is a disjoint union of copies of $A_1$, these being
permuted transitively by $\theta$. When $2\beta \in R_{\res}(G,T)$, the
Dynkin diagram of $M_\beta$ is a disjoint union of copies of $A_2$, these
being permuted transitively by $\theta$; moreover, if there are $r$ copies
of $A_2$, then $\theta^r$ preserves each copy and acts nontrivially on it.  
\end{enumerate}
\end{theorem}
\begin{proof}
See \cite{St}. 
\end{proof}

\subsection{Comparison  of twisted $a$-data with the $a$-data  used in
\cite{KS}} Part (4) of Steinberg's theorem allows us to view twisted
$a$-data $\{a_\alpha\}$ for $T$ as $a$-data $\{a_\beta\}$ for the restricted
root system, the two viewpoints being related by the equalities
$a_{\alpha_{\res}}=a_\alpha$. We say that $a$-data for $R_{\res}(G,T)$ are
\emph{special} if 
\[
a_{2\beta}=a_\beta 
\]
 whenever both $\beta$, $2\beta$ are  restricted roots. Only special
$a$-data were considered in \cite{KS} (though the word ``special'' was not
used there). Now that we have introduced twisted splitting invariants, we
may as well allow arbitrary $a$-data. Indeed, it would be awkward to compare
$\Delta_I^{\new}$ with $\Delta_I$ without doing so, as the first part of
 Proposition \ref{lem.main} can only be formulated using the
non-special  twisted 
$a$-data $\{\tilde a_\alpha\}$ appearing  there. 

\subsection{Review of $\Delta_I$} \label{sub.revDI} 
In this subsection we assume that $k$ is a local field $F$ of characteristic
$0$, so that $\Delta_I$ is defined. We now review the relevant definitions. 
Because $G$ is quasi-split and $\theta$ preserves our given $F$-splitting,
we should use the $\Delta_I$ specified in section 5.3 of \cite{KS}. In other
words, when we use the definition in 
\cite[4.2]{KS} to form $\Delta_I$, we should use the
$F$-splitting of $G^\theta$ obtained from our given $F$-splitting of $G$. 
  
We now recall the definition of  this $F$-splitting of $G^\theta$.    The
group $G^\theta$ is quasi-split, connected, reductive, with maximal torus
$\mathbf T^\theta$ and Borel subgroup $\mathbf B^\theta$, and the root
system $R(G^\theta,\mathbf T^\theta)$ is the set of indivisible roots in 
$R_{\res}(G,\mathbf T)$. The set of simple roots in $R(G^\theta,\mathbf
T^\theta)$ coincides with the set of simple roots in $R_{\res}(G,\mathbf T)$.
We complete the pair $(\mathbf B^\theta,\mathbf T^\theta)$ to an
$F$-splitting
$(\mathbf B^\theta,\mathbf T^\theta,\{X_\beta\})$ of $G^\theta$ by putting 
\[
X_\beta:=\sum_{\alpha \in  \pi^{-1}(\beta)} X_\alpha
\] 
for every simple root $\beta$ of $R_{\res}(G,T)$. 

We now choose special $a$-data $\{a_\beta\}$ on $R_{\res}(G,T)$. 
In \cite{KS} the term $\Delta_I$ was defined by 
\[
\Delta_I=\langle \lambda_{\{a_\beta\}}(T^\theta),s_{T,\theta} \rangle
\]
for a certain element $s_{T,\theta} \in (\hat T)^\Gamma_{\theta}$ that will
be discussed further in the next subsection. The pairing $\langle
\cdot,\cdot
\rangle$ is the Tate-Nakayama pairing between
$H^1(F,T^\theta)$ and $(\hat T)^\Gamma_\theta$.  
Note that $(\hat T)_\theta$ is 
Langlands dual to $T^\theta$. 
 The splitting invariant occurring in this
definition is the untwisted one from
\cite{LS} for the group $G^\theta$, and is formed using the given
$a$-data on
$R(G^\theta,T^\theta) \subset R_{\res}(G,T)$ and the $F$-splitting of
$G^\theta$ specified above. It is worth noting that no information is lost
when one restricts special $a$-data to the subset $R(G^\theta,T^\theta)$ of
indivisible roots in $R_{\res}(G,T)$

\subsection{Definition of $\Delta_I^{\new}$} In this section $k$ is a local
field $F$ of arbitrary characteristic. We consider twisted $a$-data
$\{a_\alpha\}$ for $R(G,T)$ (equivalently, $a$-data $\{a_\beta\}$ for
$R_{\res}(G,T)$) and use it to form the twisted splitting invariant
$\lambda_{\{a_\alpha\}}(T,\theta)$. We then put 
\[
\Delta_I^{\new}=\langle \lambda_{\{a_\alpha\}}(T,\theta),s_{T,\theta} \rangle
\] 
The element $s_{T,\theta}$ is the same as the one used to define $\Delta_I$.
The term $\Delta_I^{\new}$ depends on the  choice of $\theta$-invariant
$F$-splitting $(\mathbf B,\mathbf T,\{X_\alpha\})$ as well as the choice of
twisted $a$-data. 

Our next goal is to understand the dependence of $\Delta_I^{\new}$ on the
choice of twisted $a$-data. Before doing so we must  review some  material
from 
\cite{LS} and \cite{KS}. 
As in \cite{KS} there are three types of restricted roots: 
\begin{enumerate}
\item type $R_1$, for which neither $2\beta$ nor $\frac{1}{2}\beta$ is a
root, 
\item type $R_2$, for which $2\beta$ is a root, 
\item type $R_3$, for which $\frac{1}{2}\beta$ is a root.
\end{enumerate} 
The indivisible roots are  the ones of types $R_1$ and $R_2$. 

As in \cite{LS} and \cite{KS} there are two kinds of orbits $\mathcal O$ of
$\Gamma$ in $R_{\res}(G,T)$: 
\begin{enumerate}
\item symmetric, for which $\beta \in \mathcal O \implies -\beta \in
\mathcal O$, 
\item asymmetric, for which $\beta \in \mathcal O \implies -\beta \notin
\mathcal O$.
\end{enumerate} 
For $\beta$ lying in a symmetric orbit the field of definition $F_\beta$ of
$\beta$ is a separable quadratic extension of the field of definition
$F_{\pm\beta}$ of $\pm\beta$. We denote by $\sgn_{F_\beta/F_{\pm\beta}}$ the
sign character on $F^\times_{\pm\beta}$ associated by local class field
theory to the quadratic extension $F_\beta/F_{\pm\beta}$. 

We also need to review $s_{T,\theta}$. It comes from twisted endoscopic
data $(H,\dots)$, but for the present purposes  we only need to know 
what it means to say that a restricted root \emph{comes from $H$}.  
To understand this point one  must remember that the coroot system for a 
twisted endoscopic group $H$ is a subsystem of the restricted root system 
of the Langlands dual group  $\hat G$. The element $s_{T,\theta}$ tells us
what this subsystem is, in a way that we will now  recall. For this 
the reader may find it helpful to consult the discussion of twisted
centralizers on page 16 of \cite{KS}.

We need one more piece of notation. For
$\alpha^\vee \in R^\vee(G,T)$ we denote by $N(\alpha^\vee) \in X^*(\hat T)$
the sum of the elements in the $\theta$-orbit of $\alpha^\vee$. Then
$N(\alpha^\vee) \in X^*(\hat T)^\theta = X^*((\hat T)_\theta)$, which is to
say that $N(\alpha^\vee)$ may be viewed as a character on $(\hat T)_\theta$. 

The coroot system of $(H,T_\theta)$ can be identified with a certain subset
of the set of $\theta$-orbits in the coroot system $R^\vee(G,T)$. The
$\theta$-orbit of $\alpha^\vee \in R^\vee(G,T)$ lies in this subset when one
of the following two conditions holds 
\begin{enumerate}
\item $\alpha$ is of type $R_1$ or $R_2$, and
$(N(\alpha^\vee))(s_{T,\theta})=1$, or 
\item $\alpha$ is of type $R_3$ and $(N(\alpha^\vee))(s_{T,\theta})=-1$.
\end{enumerate}
 As in \cite{KS} we say that a
restricted root $\beta=\alpha_{\res}\in R_{\res}(G,T)$ \emph{comes from $H$}
when $\alpha^\vee$ satisfies either (1) or (2); this condition is obviously
independent of the choice of $\alpha$ such that $\alpha_{\res}=\beta$. 

Now we are almost ready to explain how $\Delta_I^{\new}$ depends on the
choice of twisted $a$-data. We continue to view
twisted $a$-data as being $a$-data $\{a_\beta\}$ for the restricted root
system $R_{\res}(G,T)$, and we do not assume that it is special. 
We want to see how $\Delta_I^{\new}$ changes when $\{a_\beta\}$ is replaced
by another choice $\{a'_\beta\}$ of $a$-data on $R_{\res}(G,T)$.  

Write $a'_\beta=a_\beta b_\beta$ and note that $b_\beta \in
F^\times_{\pm\beta}$. Thus the sign
$\sgn_{F_\beta/F_{\pm\beta}}(b_\beta)$ is defined whenever $\beta$ lies in a
symmetric $\Gamma$-orbit. 

\begin{lemma}\label{lem.aa'} 
When $\{a_\beta\}$ is replaced by $\{a_\beta'\}$ the term $\Delta_I^{\new}$
is multiplied by the sign 
\begin{equation}
\prod_\beta \sgn_{F_\beta/F_{\pm\beta}}(b_\beta) 
\end{equation}
where the product is taken over a set of representatives for the symmetric
$\Gamma$-orbits in the set of elements $\beta \in R_{\res}(G,T)$ that
satisfy one of the following two conditions:
\begin{enumerate}
\item $\beta$ is of type $R_3$ and comes from $H$, or 
\item $\beta$ is not of type $R_3$ and does not come from $H$. 
\end{enumerate} 
Note that the set of elements satisfying one of these two conditions is
$\Gamma$-stable and hence  a union of
$\Gamma$-orbits. 
\end{lemma} 
\begin{proof}
We claim that, when $\alpha_{\res}$ lies in a symmetric
$\Gamma$-orbit in $R_{\res}(G,T)$, the value of $N(\alpha^\vee)$ on
$s_{T,\theta}$ is $\pm 1$. Indeed, there exists $\sigma \in \Gamma$ such
that $\sigma\alpha_{\res}=-\alpha_{\res}$, and so from the
$\Gamma$-invariance of $s_{T,\theta}$ it follows that 
\[
(N(\alpha^\vee))(s_{T,\theta})=(-N(\alpha^\vee))(s_{T,\theta}). 
\] 
This just says that the square of $(N(\alpha^\vee))(s_{T,\theta})$ is
$1$, as claimed. 

The method of proof of Lemmas 3.2.C and  3.2.D in \cite{LS} applies
without change to show that 
$\Delta_I^{\new}$ is multiplied by 
\[
\prod_\beta \sgn_{F_\beta/F_{\pm\beta}}(b_\beta) 
\]
 where the product is taken over a set of representatives for the symmetric
$\Gamma$-orbits in the set 
\[
\{\alpha_{\res} \in R_{\res}(G,T): (N(\alpha^\vee))(s_{T,\theta})=-1\}. 
\]
Glancing back at what it means for $\beta$ to come from $H$, we see that the
lemma has been proved. 
\end{proof} 

\subsection{Comparison of $\Delta_I^{\new}$ and $\Delta_I$} 
In this subsection we work over a local field $F$ of characteristic $0$. 
Let us fix special $a$-data $\{a_\beta\}$ on $R_{\res}(G,T)$. For such
$a$-data both $\Delta_I$ and $\Delta_I^{\new}$ are defined, and our next goal
is to compute their ratio. 

Now $\Delta_I$ involves a splitting invariant for $T^\theta$, while
$\Delta_I^{\new}$ involves a twisted splitting invariant for $(T,\theta)$.
The former involves the liftings $n(\omega)$ for the Weyl group of
$G^\theta$, the latter  the liftings $n(\omega)$ for the 
$\theta$-fixed points  in the Weyl group of $G$. Our first task is to
compare  these two liftings, and to do so we need notation that keeps track
of which group we are using. 

For $\omega \in \mathbf\Omega$ we  write $n(\omega) \in \Norm(\mathbf
T,G)$ for the lifting of $\omega$ provided by our $\theta$-stable
$F$-splitting $(\mathbf B,\mathbf T,\{X_\alpha\})$. For $\omega \in
\mathbf\Omega^\theta$, which we also view as the Weyl group of $\mathbf
T^\theta$ in
$G^\theta$, we write $n'(\omega) \in \Norm(\mathbf T^\theta,G^\theta)$ for
the lifting of $\omega$ provided by the $F$-splitting $(\mathbf
B^\theta,\mathbf T^\theta,\{X_\beta\})$ specified in subsection
\ref{sub.revDI}.  The next lemma is valid for any field of characteristic
$0$. To formulate the lemma we need a definition: for $\alpha \in R(G,\mathbf
T)$ we put 
\[
b_\alpha:=
\begin{cases}
\frac{1}{2} &\text{ if $\alpha_{\res}$ has type $R_3$},\\ 
1 &\text{ otherwise} 
\end{cases}
\]

\begin{lemma}\label{lem.nn'} 
Let $\omega \in \mathbf \Omega^\theta$. Then both liftings of $\omega$ are
defined, and they are related by 
\[
n'(\omega)=\bigl(\prod_{\alpha \in R(\omega)}b_\alpha^{\alpha^\vee}\bigr)
n(\omega),
\] 
where $R(\omega)=\{\alpha \in R(G,\mathbf T): \alpha >
0,\omega^{-1}\alpha<0\}$. 
\end{lemma}
\begin{proof}
An easy argument using reduced expressions shows that it is enough to prove
this for simple reflections in
$\mathbf\Omega^\theta$. Then, by part (6) of Steinberg's theorem, the lemma
reduces to a calculation in root systems of type $A_1$ and $A_2$. The case
of $A_1$ is trivial. The case of $A_2$ is easy but 
interesting. We have included the calculation in an appendix. 
\end{proof}

Now we can formulate the main result of this section. 
\begin{proposition}\label{lem.main}
\hfill 
\begin{enumerate}
\item Using our given special $a$-data $\{a_\beta\}$, we define  another set 
$\{\tilde a_\beta\}$ of $a$-data on
$R_{res}(G,T)$ by the rule 
\begin{equation*}
\tilde a_\beta=
\begin{cases}
\frac{1}{2}a_\beta &\text{ if $\beta$ has type $R_3$},\\ 
a_\beta &\text{otherwise}. 
\end{cases}
\end{equation*} 
 We then have the equality 
\[\lambda_{\{a_\beta\}}(T^\theta)=\lambda_{\{\tilde
a_{\beta}\}}(T,\theta).\] 
Notice that $\{\tilde a_\beta\}$ is non-special when restricted roots of type
$R_3$ exist.
\item There is an equality 
\begin{equation*}
\frac{\Delta_I^{\new}}{\Delta_I}=\prod_\beta
\sgn_{F_\beta/F_{\pm\beta}}(2)
\end{equation*}
where the product is taken over a set of representatives $\beta$ for the
symmetric
$\Gamma$-orbits in the set of restricted roots having type $R_3$ and coming
from $H$.
\end{enumerate}
\end{proposition} 

 Consequently, replacing $\Delta_I$ by $\Delta_I^{\new}$ has 
the same effect as modifying $\Delta_{II}$ in the way proposed by
Waldspurger \cite{W2}. 

\begin{proof}
(2) follows from (1) and Lemma \ref{lem.aa'}, so it suffices to prove (1). 

In order to define the twisted splitting invariant we need to choose $h \in
G^\theta(\bar F)$ such that $(B,T)^h=(\mathbf B,\mathbf T)$. We then have
$(B^\theta,T^\theta)^h=(\mathbf B^\theta,\mathbf T^\theta)$, so $h$ also
serves to define the untwisted splitting invariant for $T^\theta$. 

Our task is then to compare  the $1$-cocycle
$t(\sigma)$ (see section \ref{sec.2}) 
 obtained from $\{\tilde a_\beta\}$ and $(G,T,\theta)$ with the
$1$-cocycle $t'(\sigma)$ obtained from $\{a_\beta\}$ and
$(G^\theta,T^\theta)$. We just need to show that the two cocycles are
cohomologous. Because we use the same element $h$ to define
both, they will turn out to be equal. 

From \eqref{eq.zeta}, \eqref{eq.m}, \eqref{eq.t} we see that it suffices to
prove that $m(\sigma)=m'(\sigma)$, where 
\[
m(\sigma) = \bigl(\prod_{\alpha \in R(\sigma_T)}
{\tilde a}_{\alpha_{\res}}^{\alpha^\vee}\bigr) n(\omega_T(\sigma))
\] 
and 
\[
m'(\sigma) = \bigl(\prod_{\beta \in R'(\sigma_{T^\theta})}
a_{\beta}^{\beta^\vee}\bigr) n'(\omega_T(\sigma));
\] 
here $R'$ is the root system $R(G^\theta,\mathbf T^\theta)$, which we
identify with the set of indivisible roots in $R_{\res}(G,T)$, and
$\beta^\vee$ is the coroot for the group $G^\theta$ associated to $\beta \in
R'$. As in  section \ref{sec.2} we  transport $a$-data from $T$ to
$\mathbf T$  without change of notation. 

For $\beta \in R(G^\theta,\mathbf T^\theta)$ a computation in root systems of
types
$A_1$ and
$A_2$ shows that
$\beta^\vee
\in X_*(T^\theta)=X_*(T)^\theta$ is as follows. 
Choose $\alpha
\in R(G,T)$ such that $\alpha_{\res}=\beta$; then 
\[
\beta^\vee=
\begin{cases}
N(\alpha^\vee) &\text{ if $\beta$ is of type $R_1$},\\
2N(\alpha^\vee) &\text{ if $\beta$ is of type $R_2$.}
\end{cases}
\] 
From this it follows easily that 
\[
\prod_{\alpha \in R(\sigma_T)}
a_{\alpha_{\res}}^{\alpha^\vee}=\prod_{\beta \in R'(\sigma_{T^\theta})}
a_{\beta}^{\beta^\vee}.
\] 
Therefore it suffices to show that 
\begin{equation}\label{eq.n'bn}
n'(\omega_T(\sigma))=\bigl( \prod_{\alpha \in
R(\sigma_T)} b_{\alpha}^{\alpha^\vee}\bigr) n(\omega_T(\sigma))
\end{equation}
where $b_\alpha$ is defined by 
\[
b_\alpha=
\begin{cases}
\frac{1}{2} &\text{ if $\alpha_{\res}$ has type $R_3$},  \\
1 &\text{ otherwise.}
\end{cases}
\] 
Now $R(\sigma_T)=R(\omega_T(\sigma))$; this is  
 an immediate consequence of the fact that $\sigma_{\mathbf
T}$ preserves the set of positive roots. Therefore the
equality \eqref{eq.n'bn} follows from Lemma \ref{lem.nn'}.
\end{proof}

\section{Two normalizations of the local Langlands
correspondence}\label{sec.ren}

\subsection{Two normalizations of the  reciprocity law isomorphism}
 Consider a nonarchimedean local field $F$. 
There are two ways to normalize the reciprocity law isomorphism $F^\times \to
W_F^{\ab}$.  In the classical version uniformizers correspond to the
Frobenius automorphism,  and in Deligne's
version  uniformizers correspond to the inverse of the Frobenius
automorphism.  Tate's article \cite{Tate} makes use of Deligne's
normalization, while Borel's article \cite{Borel} makes use of the classical
one. 

\subsection{Two normalizations of the local Langlands correspondence for
tori} More generally, there are two  ways to normalize the Langlands
correspondence for tori $T$ over $F$. Borel, following Langlands's 
conventions \cite{L}, uses the version which   is
compatible with the classical  reciprocity law when  $T=\mathbb G_m$. In
order to obtain a version of the Langlands correspondence for tori that is
compatible with Deligne's normalization of the reciprocity law, one has
simply to build in an inverse. In other words, if a quasicharacter $\chi$ on
$T(F)$ corresponds to a Langlands parameter
$\varphi:W_F
\to {}^LT$ under the classical Langlands correspondence for tori, then it
is $\chi^{-1}$ that corresponds to $\varphi$ in the version compatible with
Deligne's conventions. 

If one wants to follow Deligne's conventions for nonarchimedean local
fields, and one wants to have local-global compatibility, one  is forced to
build in an inverse when one consider the Langlands correspondence for tori
over global fields, and then one is forced to build it in for
$\mathbb R$ and $\mathbb C$ as well. 

\subsection{Two normalizations of the Langlands pairing}\label{sub.2norms}
Let $F$ be a local field.  
The two versions of the Langlands correspondence for tori $T$ over $F$ give
rise to two versions of the Langlands pairing between $T(F)$ and
$H^1(W_F,\hat T)$. We need a system of notation that  distinguishes between
them. Let $t \in T(F)$ and $\mathbf a \in H^1(W_F,\hat T)$. The classical
Langlands pairing is defined by 
$\langle t, \mathbf a \rangle:=\chi(t)$, where $\chi$ is the quasi-character
on $T(F)$ corresponding to $\mathbf a$ under the classical Langlands
correspondence for tori. 

We define a renormalized Langlands pairing by  
\[
\langle
t, \mathbf a \rangle_D:=\langle
t, \mathbf a \rangle^{-1}.
\] 
The subscript is meant to remind us that the version $\langle
t, \mathbf a \rangle_D$ of the Langlands pairing is 
the one compatible with Deligne's normalization of the reciprocity law. 

Now let $G$ be a  connected reductive group over $F$. In this setting
too  there is  a Langlands pairing as well as a renormalized version of
it.  For
$g \in G(F)$ and $\mathbf a \in H^1(W_F,Z(\hat G))$ we denote by 
$\langle g, \bold a \rangle$ the classical Langlands pairing between $g$ and
$\mathbf a$. We then define a renormalized version by putting 
\[
\langle
g, \mathbf a \rangle_D:=\langle
g, \mathbf a \rangle^{-1}.
\] 
In the case where $G$ is a torus, these two pairings coincide with the ones
we have just discussed.  

In this system of notation (A.3.13) on p.~137 of \cite{KS} becomes 
\[
\langle j(u),\hat z\rangle=\langle u,\hat i(\hat z) \rangle_D^{-1}=
\langle u,\hat i(\hat z) \rangle.
\]

\subsection{Two normalizations of the local Langlands correspondence for
connected reductive groups} 
 Still more generally there should be two ways to
normalize the conjectural Langlands correspondence for arbitrary connected
reductive groups over a local field $F$. There should be one version that is
compatible with the classical Langlands correspondence for tori. It should
also be compatible with the classical version of the local Langlands
correspondence for unramified representations \cite{Borel}. There should be
another version that is compatible with the inverse normalization of the
Langlands correspondence for tori, and also with the local Langlands
correspondence for unramified representations  defined using the geometric
Frobenius in place of the (arithmetic) Frobenius automorphism. Number
theorists, when considering the local Langlands correspondence for $GL_n$,
tend to use this second version.

These  two versions of the local Langlands correspondence should be related
in the following way. The $L$-group of a connected reductive group $G$ is
defined as a semidirect product $\hat G \rtimes W_F$. The action of $W_F$
on $\hat G$ preserves some splitting $(\mathcal B,\mathcal T,\{\chi\})$ (in
the notation of \cite{KS}). There is a unique automorphism $\theta_0$ of
$\hat G$ that preserves  $(\mathcal B,\mathcal T,\{\chi\})$ and on $\mathcal
T$ induces the automorphism $t \mapsto \omega_0(t)^{-1}$, where $\omega_0$
is the longest element in the Weyl group of $\mathcal T$. The automorphism
$\theta_0$ commutes with the action of $W_F$, and therefore we obtain an
automorphism ${}^L\theta_0$ of ${}^LG$ defined by     
${}^L\theta_0(g\sigma)=\theta_0(g)\sigma$ for $g \in \hat G$, $\sigma \in
W_F$. Let $\varphi: W_F \to {}^LG$ be a  Langlands parameter, and let
$\Pi$ be the (conjectural) $L$-packet attached to $\varphi$ by the
classical normalization of the local Langlands correspondence. Then in the
other version of the Langlands correspondence  that same packet
$\Pi$ should be attached to the Langlands parameter ${}^L\theta_0 \circ
\varphi$ obtained by composing ${}^L\theta_0$ with $\varphi$. Notice that
for tori this procedure does agree with the one described earlier, because
$\theta_0$ is simply the inversion map on $\mathcal T$. For $GL_n$ a
Langlands parameter
$\varphi$ amounts to an $n$-dimensional representation of $W_F$, and
replacing
$\varphi$ by ${}^L\theta_0\circ\varphi$ amounts to replacing that
$n$-dimensional representation by its contragredient. The situation in
general is entirely analogous,  because, for any finite dimensional complex
representation $r$ of
${}^LG$, the representation $r \circ {}^L\theta_0$ is isomorphic to the
contragredient of $r$. Thus, in a global context, replacing $\varphi$ by
${}^L\theta_0\circ\varphi$ amounts to replacing the arithmetic Frobenius by
the geometric Frobenius when defining automorphic $L$-functions. 

Our discussion may be summarized as follows. There should be two versions of
the local Langlands correspondence. The first  is the \emph{classical} one,
based on the arithmetic Frobenius. The second,  based on the
geometric Frobenius, and compatible with  Deligne's normalization of the
reciprocity law, will be referred to here as  the \emph{renormalized}
Langlands correspondence. We will write
$\Pi(\varphi)$ for the conjectural $L$-packet attached to a Langlands
parameter $\varphi$ by means of the classical Langlands correspondence, and
we will write
$\Pi_D(\varphi)$ for the one obtained using the renormalized Langlands
correspondence. Thus $\Pi_D(\varphi):=\Pi({}^L\theta_0\circ\varphi)$ (in
situations in which the classical Langlands correspondence is known).  

\section{Two versions $\Delta_D$ and $\Delta'$ of corrected twisted transfer
factors}\label{sec.55} 
\subsection{Standard endoscopy}\label{sub.std} 
One goal of standard endoscopy is to provide character identities
associated to endoscopic data $(H,s,\xi)$. Recall that $\xi$ is an
$L$-homomorphism $\xi:{}^LH \to {}^LG$. Given a tempered Langlands parameter
$\varphi_H:W_F \to {}^LH$ for $H$, one forms the Langlands parameter
$\varphi:=\xi
\circ \varphi_H$ for $G$, and then one expects to have character identities
involving the members of $\Pi(\varphi_H)$ on one side, and the members of
$\Pi(\varphi)$ on the other. The transfer factors in \cite{LS} are expected
to produce a notion of endoscopic transfer yielding such character
identities.

If, however, one prefers to use the renormalized local
Langlands correspondence, then one needs to renormalize the transfer factors
 $\Delta=\Delta_I\Delta_{II}\Delta_{III_1}\Delta_{III_2}\Delta_{IV}$ of
\cite{LS}. The renormalized factors, that we will denote by
$\Delta_D$, are easy to define. Only the term
$\Delta_{III_2}$ is affected, which is to say that $\Delta_D$ takes the form 
$\Delta_I
\Delta_{II}\Delta_{III_1}\Delta_{III_2,D}\Delta_{IV}$. Moreover the
renormalized term  $\Delta_{III_2,D}$  is formed as follows. Recall that 
$\Delta_{III_2}=\langle
\gamma,\mathbf a \rangle$ (see \cite[p.~247]{LS}). We now define 
$\Delta_{III_2,D}$ to be $\langle \gamma,\mathbf a_D\rangle_D$, but we need
to explain the notation used in  this expression.   

The renormalized Langlands pairing $\langle \cdot,\cdot \rangle_D$ was
defined in subsection
\ref{sub.2norms}; recall that    
 $\langle\gamma,\mathbf a_D
\rangle_D=\langle\gamma, \mathbf a_D\rangle^{-1}$. The relationship
between $\mathbf a$ and $\mathbf a_D$ is a bit more subtle. Recall from  
\cite[\S 3.5]{LS} that $\mathbf a \in H^1(W_F,\hat T)$  depends on the
choice of
$\chi$-data $\chi=\{\chi_\alpha\}$, and so we should  write $\mathbf a(\chi)$
when we need to keep track of this dependence.  Now
$\{\chi_\alpha^{-1}\}$ are also
$\chi$-data and we put $\mathbf a_D(\chi)=\mathbf a(\chi^{-1})$, where
$\chi^{-1}$ is an abbreviation for $\{\chi_\alpha^{-1}\}$. In this system of
notation we have 
\begin{align}
\Delta_{III_2}&=\langle \gamma,\mathbf a(\chi)\rangle,\\
\Delta_{III_2,D}&=\langle \gamma,\mathbf a_D(\chi)\rangle_D=\langle
\gamma,\mathbf a(\chi^{-1})\rangle^{-1}.
\end{align}

There is yet another way in which  transfer factors can be modified. 
For this one needs to notice that if $(H,s,\xi)$ is endoscopic data, so too
is $(H,s^{-1},\xi)$. We  obtain transfer factors $\Delta$ and $\Delta_D$
from 
$(H,s,\xi)$, and we also obtain transfer factors $\Delta'$ and $\Delta'_D$
from $(H,s^{-1},\xi)$. However, as in \cite{AA}, we can take a different
point of view by regarding $\Delta'$ as an alternative version of transfer
factors for the original endoscopic data $(H,s,\xi)$. The same applies
to $\Delta_D'$, and altogether there are four useful variants of transfer
factors for $(H,s,\xi)$, namely $\Delta$, $\Delta'$, $\Delta_D$ and
$\Delta_D'$. It will be necessary to bear this in mind as we turn now to 
twisted transfer factors. 

\subsection{Twisted endoscopy}
One goal of twisted endoscopy is  to provide twisted character
identities, and therefore one should expect to have twisted transfer
factors  adapted to the renormalized Langlands correspondence as well
as ones adapted to the classical Langlands
correspondence. We are now going to elaborate on this point, and at the
same time correct an error  in \cite{KS} that was found by Waldspurger. 

The error in \cite{KS} arose from  being inconsistent about  the
normalization of the Langlands correspondence for tori. The appendices of
\cite{KS} make use of the renormalized version of the Langlands
correspondence for tori, as one sees from the presence of $a^{-1}$ rather
than $a$ in the expression 
\[
\prod_{a \in K^\times} x_a(a^{-1})
\]
displayed in the middle of page~131 of \cite{KS}. However, in the course of
defining  twisted $\Delta_{III}$ (see p.~40 of \cite{KS}) we made use of the
admissible embeddings $^LT_H \hookrightarrow {}^LH$ and ${}^L(T_{\theta^*})
\hookrightarrow {}^LG^1$ obtained by applying \cite[\S
2.6]{LS} to our chosen $\chi$-data. This leads to nonsense because the
construction in \cite{LS} is adapted to the classical Langlands
correspondence. Indeed, as Waldspurger pointed out to us, the dependence
on $\chi$-data of the terms  $\Delta_{II}$ and $\Delta_{III}$ defined in
\cite{KS} does not cancel in the way that it should. (We of course want
twisted transfer factors to be independent of the choice of $\chi$-data.) 

There are two ways to fix this error. One leads to twisted transfer 
factors 
$\Delta'$ adapted to the classical Langlands
correspondence, and one leads to twisted transfer factors $\Delta_D$ adapted
to the renormalized Langlands correspondence.  We begin with twisted
$\Delta_D$, which requires only a small modification of \cite{KS}.   

\subsection{Twisted $\Delta_D$} 
We will define twisted $\Delta_D$ to be the product
$\Delta_I^{\new}\Delta_{II}\Delta_{III}^{\new}\Delta_{IV}$. In this product
$\Delta_{II}$ and $\Delta_{IV}$  are the terms defined in \cite{KS}, and
$\Delta_I^{\new}$ is the modified version of $\Delta_I$ that was defined 
earlier in this note. The only new term is $\Delta_{III}^{\new}$, which we
will now explain. 

The term $\Delta_{III}$ in \cite{KS} was defined in section 4.4 of
\cite{KS}. The definition was given first in the special case
in which
$H_1=H$ and was then given in the general case. Moreover, in section 5.3 of
\cite{KS} a simpler version of $\Delta_{III}$ was defined in the special case
when $G$ is quasisplit and $\theta$ preserves an $F$-splitting. All three
definitions follow the same pattern, and the modification needed to obtain
$\Delta_{III}^{\new}$ is the same in all three cases.  The modification is
easy to explain. The term $\Delta_{III}$ of \cite{KS} depends on a choice of
$\chi$-data. We now define $\Delta_{III}^{\new}$ to be the term
$\Delta_{III}$ for the \emph{inverse} set of
$\chi$-data. (As in \ref{sub.std} the inverse set of $\chi$-data is obtained
by replacing each
$\chi_{\alpha_{\res}}$ by $\chi_{\alpha_{\res}}^{-1}$.) 

There is an equivalent way to define twisted $\Delta_D$, due to the fact that
$\Delta_D$ is independent of the choice of $\chi$-data. Replacing our given
set of $\chi$-data by its inverse, we obtain the equality 
\begin{equation}
\Delta_D=\Delta_I^{\new}\Delta_{II}^{-1}\Delta_{III}\Delta_{IV}. 
\end{equation}
We used the obvious fact that replacing $\chi$-data by its inverse replaces
$\Delta_{II}$ by its inverse.

 The twisted transfer factor $\Delta_D$ has all the
properties stated in \cite{KS}. No further modifications need be made to
\cite{KS} (as far as we know), though one must always remember to interpret 
the phrases ``the Langlands correspondence'' and ``the Langlands pairing'' as
referring to the renormalized versions.  For example the quasicharacter
$\omega$  on p.~17 of \cite{KS} needs to be defined by
$\omega(g)=\langle g,\mathbf a \rangle_D$. Similarly, the quasicharacter
$\lambda_{H_1}$  of \cite{KS} needs to be defined by 
$\lambda_{H_1}(z_1)=\langle z_1,\mathbf b \rangle_D$, where $\mathbf b$ is
the Langlands parameter
\[
W_F\xrightarrow{c}\mathcal H \xrightarrow{\xi_{H_1}}{}^LH_1 \to {}^LZ_1
\]
considered on p.~23 of \cite{KS}. 

\begin{remark}
In the case of standard endoscopy, the Langlands parameter $\mathbf b$ above 
is
\emph{inverse}, in the group
$H^1(W_F,\hat Z_1)$, to the Langlands parameter $W_F \to {}^LG_1 \to
{}^LZ_1$ defined on p.~253 of
\cite{LS}. This remark clarifies the meaning of the comment made on lines
-4,-5 on p.~23 of \cite{KS}. 
\end{remark}

In the case of standard endoscopy the twisted factor $\Delta_D$ reduces to
the factor 
$\Delta_D$ discussed in the previous subsection, which explains our choice
of notation.   It is also desirable to have twisted transfer
factors adapted to the classical Langlands correspondence.  In the next
subsection we will see how to define twisted transfer factors $\Delta'$ that
reduce to   the factor
$\Delta'$ discussed in  the previous subsection in the case of standard
endoscopy. 

\subsection{Twisted $\Delta'$} Following a suggestion made to us by 
Waldspurger (and incorporating the improved version $\Delta_I^{\new}$ of
twisted
$\Delta_I$ discussed earlier in this note), we now put  
\begin{equation}
\Delta':=(\Delta_I^{\new}\Delta_{III})^{-1}\Delta_{II}\Delta_{IV}.
\end{equation}
In
this product  $\Delta_{II}$, $\Delta_{III}$ and $\Delta_{IV}$  are the terms
defined in
\cite{KS}. 

If one  uses the twisted transfer factors $\Delta'$, then one needs to 
insert many minus signs  in \cite{KS}, especially in the part concerning the
stabilization of the twisted trace formula.  A complete list of the necessary
changes would be rather long.  We will discuss only the  most significant 
ones. 

The first point to make is that now the quasicharacters $\omega$ and
$\lambda_{H_1}$ need to be defined using the classical Langlands pairing.  
 In other words, the quasicharacter 
$\omega$  on p.~17 of \cite{KS} now needs to be defined by
$\omega(g)=\langle g,\mathbf a \rangle$. Similarly, the quasicharacter
$\lambda_{H_1}$  of \cite{KS} now  needs to be defined by 
$\lambda_{H_1}(z_1)=\langle z_1,\mathbf b \rangle$, with $\mathbf b$ as in
the previous subsection. With this understanding Lemma 5.1.C of \cite{KS} is
correct for $\Delta'$, as is part (2) of Theorem 5.1.D. However part (1) of
that theorem needs to be replaced by 
\[
\Delta'(\gamma_1,\delta')=\langle \inv(\delta,\delta'),\kappa_\delta
\rangle^{-1}\Delta'(\gamma_1,\delta).
\]

When one uses $\Delta'$, one must insert many minus signs  in the
stabilization of the elliptic $\theta$-regular terms on the geometric side of
the  twisted trace formula. We will now discuss 
all  the basic definitions and key statements that need to be modified. These
modifications entail changes in some of the formulas in the proofs, but only
a few of these are listed here. The rest will be obvious to anyone who 
systematically works through the proof of the stabilization.  
All the page references in the rest of this subsection are to
\cite{KS}. 

The equality in the fourth line on p.~81  must be replaced by 
\[\langle \inv'(\delta,\delta'),\beta(\mathbf a)\rangle^{-1}=\omega(h),\] 
and the equality on the ninth and tenth lines of that page must be replaced
by 
\[O_{\delta'\theta}(f)=\langle\inv'(\delta,\delta'),\beta(\mathbf a) \rangle
O_{\delta\theta}(f).\] There is no change in the displayed formula (6.2.2)
on that page. 

On p.~89 the definition of $\Phi$ must be replaced by 
\[
\Phi(x)=\langle \obs(\delta),\kappa_0 \rangle O_{\delta\theta}(f). 
\] The displayed formula in line -7 of that page must be replaced by 
\[
\langle\inv(\delta,\delta'),\kappa_0\rangle^{-1}=\omega(h).
\]

On p.~93 the expression (6.4.8) must be replaced by 
\[
\langle\inv(\delta_0,\delta(v)),\kappa_0\rangle^{-1}O_{\delta(v)\theta}(f_v).
\]
On p.~94 the first factor in the expression displayed on line 12 should be
replaced by $\langle\inv(\delta_0,\delta),\kappa_0\rangle^{-1}$. 

On p.~96 the definition of the twisted $\kappa$-orbital integral should be
replaced by 
\[
O^\kappa_{\delta_0\theta}(f)=\int_{\mathcal D(T,\theta,\mathbb A)} \langle
e,\kappa \rangle^{-1} O_{\delta_e\theta}(f) \,de_{\Tam}. 
\]
The factor $\langle
e,\kappa \rangle$ again needs to be replaced by $\langle
e,\kappa \rangle^{-1}$ in the expression on line 13 of p.~102. 

The last changes are especially significant. The right side of the
equality in Lemma 7.3.A on p.~109 should be replaced by its inverse (when
$\Delta'$ is used on the left side). The same is true for the equality in
Corollary 7.3.B, which must be replaced by  
\[
\Delta'_\mathbb A(\gamma_1,\delta)=\langle\obs(\delta),\kappa\rangle^{-1}. 
\]  

\subsection{Whittaker normalization of transfer factors}\label{sub.Wh}
Assume that
$G$ is quasi-split, and  choose an $F$-splitting $(B,T,\{X_\alpha\})$. 
Assume further that $\theta$ preserves the chosen $F$-splitting. In this
situation one can consider the Whittaker normalization $\Delta_\lambda$ of
twisted transfer factors introduced in section 5.3 of
\cite{KS}.  We remind the reader that $\Delta_\lambda$ was defined by 
\[
\Delta_\lambda=
\varepsilon_L(V,\psi)\Delta_I\Delta_{II}\Delta_{III}\Delta_{IV}. 
\]
(see pages 63 and 65 of
\cite{KS}). 

Of course this definition too needs to be modified. Again there are two
variants, these being Whittaker normalized versions   $\Delta_D^\lambda$, 
$\Delta'_\lambda$ of $\Delta_D$, $\Delta'$ respectively. The two variants 
are defined by 
\begin{align}
\Delta_D^\lambda&=
\varepsilon_L(V,\psi)\Delta_I^{\new}\Delta_{II}\Delta_{III}^{\new}\Delta_{IV},\\
&=\varepsilon_L(V,\psi)\Delta_I^{\new}
\Delta_{II}^{-1}\Delta_{III}\Delta_{IV},\notag \\ 
\Delta'_\lambda&=
\varepsilon_L(V,\psi)(\Delta_I^{\new}\Delta_{III})^{-1}\Delta_{II}\Delta_{IV}.
\end{align} 

The square of  $\epsilon_L(V,\psi)$ is $\pm1$ (see p.~65 of \cite{KS}).
Therefore it may well happen that  $\epsilon_L(V,\psi)$ is not equal to its
inverse, and it should be emphasized that we use $\epsilon_L(V,\psi)$, as
opposed to its inverse, 
 in the definitions of both $\Delta_D^\lambda$ and
$\Delta'_\lambda$. 

In the next subsection we will make use of the following definition:  
\[\Delta'_0:=(\Delta_I^{\new}\Delta_{III})^{-1}\Delta_{II}\Delta_{IV}.\] 
Of course we then have the equality
$\Delta'_\lambda=\varepsilon_L(V,\psi)\Delta'_0$. In the case of standard
endoscopy $\Delta'_0$ coincides with the factor $\Delta_0$ defined on p.~248
of \cite{LS}, but with $(H,s,\xi)$ replaced by $(H,s^{-1},\xi)$. In other
words, $\Delta'_0$ bears the same relation to $\Delta_0$ as $\Delta'$ does
to $\Delta$ (see subsection \ref{sub.std}). 

\subsection{Twisted transfer factors for cyclic base change} \label{sub.cbc}
The twisted transfer factor $\Delta'_0$ considered in the last subsection
differs  in two ways from the incorrectly defined twisted transfer factor
$\Delta_0$ on p.~63 of \cite{KS}. It uses $\Delta_I^{\new}$ rather than
$\Delta_I$, and  both $\Delta_I^{\new}$ and $\Delta_{III}$
occur with an inverse.  Only the second of these two
differences is relevant in the special case of cyclic base change, since
$\Delta_I^{\new}=\Delta_I$ in that case. 

Since twisted $\Delta_0$ was incorrectly defined in \cite{KS}, it is
necessary to reformulate Proposition A.1.10 and Corollary A.2.10    
in \cite{App} by using $\Delta'_0$ in place of $\Delta_0$.  
Proposition A.1.10 in \cite{App}  should be replaced by the
following result. 
\begin{proposition} \label{prop.app}
There is an equality 
\begin{equation}\label{eq.fa}
\Delta'_0(\gamma_H,\delta)=\Delta'_0(\gamma_H,\gamma)
\langle\inv(\gamma,\delta),(a^{-1},\tilde s) \rangle^{-1}.
\end{equation} 
\end{proposition}  
The notation used in the second factor in the righthand side of
\eqref{eq.fa} is the same as in \cite{App}, but the exponent $-1$ occurring
in this factor is new. 
\begin{proof}
One just needs to work through the proof of Proposition A.1.10, making sure
to use the Langlands pairing more carefully than was done  
there. In the displayed formula $\Delta_{III}(\gamma_H,\gamma)=\langle
\gamma,b\rangle$ in the middle of p.~194 of \cite{App}, the Langlands pairing
occurring on the right side is the one used in \cite{LS}, namely the
classical one. In the formula
$\langle(\gamma,\delta),(b,1)\rangle^{-1}=\langle
\gamma,b \rangle$ on the third line from the bottom on p.~194 of \cite{App},
the Langlands pairing occurring on the right side is the renormalized one 
$\langle
\gamma,b \rangle_D$
discussed in subsection
\ref{sub.2norms}. 
\end{proof} 

Corollary A.2.10 in \cite{App}  should be replaced by the following result. 
\begin{theorem} \label{thm.app}
There is an equality 
\begin{equation}\label{eq.fb} 
\Delta'_0(\gamma_H,\delta)=\Delta'_0(\gamma_H,\gamma)
\langle\alpha(\gamma,\delta),s\rangle. 
\end{equation}
\end{theorem} 
The notation used in the second factor in the righthand side of
\eqref{eq.fb} is the same as in \cite{App}, but the exponent $-1$ occurring
in \cite{App} is no longer present. 
\begin{proof}
Combine Proposition \ref{prop.app} above with Theorem A.2.9 in \cite{App}.
\end{proof}

Fortunately Theorem \ref{thm.app} is exactly the result needed to justify
the main contention made in \cite{App}, namely that it was legitimate to
use   
$\langle \alpha(\gamma_0;\delta),s\rangle \Delta_p(\gamma_H,\gamma_0)$ as
twisted transfer factors in \cite{AA}. The point is that the factor denoted
by 
$\Delta_p(\gamma_H,\gamma_0)$ in  \cite[p.~178]{AA} coincides with the
factor denoted by $\Delta'_0(\gamma_H,\gamma_0)$ in this note. 

\subsection{Another correction} We also take this opportunity to point out
that the definition of hypercohomology groups 
 given in   \cite[A.1]{KS} is correct only for  complexes of
$G$-modules that are bounded below. (For Tate hypercohomology the complexes
even need to be bounded above and below.) This does not affect the main
results in that appendix, which only involve bounded complexes. 

\appendix 
\section{Computations in $SL_3$} \label{sec.5}
\subsection{Standard splitting and automorphism for $SL(3)$} 
Let $k$ be a field. We consider the group $G=SL(3)$ 
and the 
automorphism $\theta$ of $G$ given by 
\[
\theta(g)=
\begin{bmatrix}
0&0&1 \\
0&-1&0 \\
1 & 0 & 0
\end{bmatrix} 
{}^tg^{-1}
\begin{bmatrix}
0&0&1 \\
0&-1&0 \\
1 & 0 & 0
\end{bmatrix}^{-1}. 
\] 
The automorphism $\theta$ has order $2$ and preserves the following splitting
 $(\mathbf B,\mathbf T,\xi_1,\xi_2)$ of
$G$. As maximal torus $\mathbf T$, we take the
diagonal matrices  in $G$. As Borel subgroup $\mathbf B$, we take the upper
triangular matrices in
$G$. The simple roots $\alpha_1$, $\alpha_2$ take the values $\alpha_1=a/b$
and 
$\alpha_2=b/c$ on the diagonal matrix 
\[
\begin{bmatrix}
a&0&0 \\
0&b&0 \\
0 & 0 & c
\end{bmatrix}
\]
and the remaining positive root is $\alpha_3=\alpha_1+\alpha_2$. 
The last ingredient in the standard splitting consists of  the following two 
homomorphisms $\xi_1$, $\xi_2$ from $SL(2)$ to $G$:
\begin{align*}
\xi_1
\begin{bmatrix}
a&b \\
c&d 
\end{bmatrix} 
&= 
\begin{bmatrix}
a&b&0 \\
c&d&0 \\
0 & 0 & 1
\end{bmatrix} \\
\xi_2
\begin{bmatrix}
a&b \\
c&d 
\end{bmatrix} 
&= 
\begin{bmatrix}
1&0&0 \\
0&a&b \\
0 & c & d
\end{bmatrix}
\end{align*}
It is obvious that $\theta \mathbf T=\mathbf T$, $\theta \mathbf B=
\mathbf B$ and 
$\theta
\circ \xi_1=\xi_2$. 

We write $N$ for the normalizer of $\mathbf T$ in $G$. 
Recall that untwisted splitting invariants for $G$ are 
constructed using the   
 elements $n_1,n_2 \in N$ defined by 
\[
n_i=\xi_i 
\begin{bmatrix}
0&1 \\
-1&0 
\end{bmatrix}
\]  
 
The group of fixed points of $\theta$ in the Weyl group
$\mathbf\Omega=N/\mathbf T$ is
$\{1,\omega_0\}$, where $\omega_0$ is the longest element of
$\mathbf\Omega$.  The standard lift
of  $\omega_0$ is 
 $n_3:=n_1n_2n_1=n_2n_1n_2$, 
Note that $\theta(n_3)=n_3$, simply because $\theta$ exchanges $n_1$ and
$n_2$.
 Explicitly, we have 
\[ 
n_3=
\begin{bmatrix}
0&0&1 \\
0&-1&0 \\
1 & 0 & 0
\end{bmatrix}
 \]

\subsection{The homomorphism $SL(2) \to SL(3)$ determined by the adjoint 
representation of $SL(2)$} 
The adjoint representation of $SL(2)$ will give us a particular 
homomorphism $\Ad:SL(2) \to SL(3)$ once we choose a basis for the Lie algebra
of
$SL(2)$. We take a slightly non-traditional basis, 
namely 
\[
X=\begin{bmatrix}
0&-1 \\
0&0 
\end{bmatrix}
\qquad
H=\begin{bmatrix}
1&0 \\
0&-1 
\end{bmatrix} 
\qquad 
Y=\begin{bmatrix}
0&0 \\
1&0 
\end{bmatrix}
\]
The non-traditional minus sign occurring in the matrix defining $X$ makes the
explicit formula for $\Ad:SL(2) \to SL(3)$ free of  minus signs. This formula
is as follows:  
\[
 \begin{bmatrix}
a&b \\
c&d 
\end{bmatrix} \mapsto
\begin{bmatrix}
a^2&2ab&b^2 \\
ac&ad+bc&bd \\
c^2 & 2cd & d^2
\end{bmatrix}
\] 
The number $2$ appearing in this last $3\times 3$ matrix is the source of
the difficulty in finding a good definition of $\Delta_I$. The point is that
the adjoint representation of
$SL(2)$ is a bit pathological  in characteristic $2$. 

Now let us assume that the characteristic of $k$ is not $2$. We are then 
free to conjugate the homomorphism $\Ad$ by the diagonal matrix 
$(1,2,2)$, thus obtaining the homomorphism $\Ad':SL(2) \to SL(3)$ given by 
\[
 \begin{bmatrix}
a&b \\
c&d 
\end{bmatrix} \mapsto
\begin{bmatrix}
a^2&ab&\frac{1}{2}b^2 \\
2ac&ad+bc&bd \\
2c^2 & 2cd & d^2
\end{bmatrix}
\] 
In particular  
\[
 \begin{bmatrix}
1&x \\
0&1 
\end{bmatrix} \mapsto
\begin{bmatrix}
1&x&\frac{1}{2}x^2 \\
0&1&x \\
0 & 0 & 1
\end{bmatrix}
\] 
and
\[
 \begin{bmatrix}
0&1 \\
-1&0 
\end{bmatrix} 
\mapsto
\begin{bmatrix}
0&0&\frac{1}{2}\\
0&-1&0\\
2 & 0 & 0
\end{bmatrix}
=:n'_3
\]

Note that $\Ad'$ induces an isomorphism from $PGL(2)$ to  $G^\theta$,
the group of fixed points of
$\theta$ on
$G$. Thus $n'_3$ is used to form untwisted splitting invariants 
for $G^\theta$, and it is  interesting to compare $n'_3$ with $n_3$. 
Inspecting the two matrices, one finds that
$n_3'=(\frac{1}{2})^{\alpha_3^\vee}n_3$.

\bibliographystyle{amsalpha}
\providecommand{\bysame}{\leavevmode\hbox to3em{\hrulefill}\thinspace}

\end{document}